\documentclass[pdflatex,sn-mathphys-num]{sn-jnl}

\usepackage{graphicx}
\usepackage{multirow}
\usepackage{amsmath,amssymb,amsfonts}
\usepackage{amsthm}
\usepackage{mathrsfs}
\usepackage[title]{appendix}
\usepackage{xcolor}
\usepackage{textcomp}
\usepackage{manyfoot}
\usepackage{booktabs}
\usepackage{algorithm}
\usepackage{algorithmicx}
\usepackage{algpseudocode}
\usepackage{listings}

\usepackage{needspace}
\theoremstyle{thmstyleone}
\newtheorem{theorem}{Theorem}[section]
\newtheorem{lemma}[theorem]{Lemma}
\newtheorem{corollary}[theorem]{Corollary}

\theoremstyle{thmstylethree}
\newtheorem{definition}[theorem]{Definition}

\theoremstyle{thmstyletwo}
\newtheorem{remark}[theorem]{Remark}

\begin{document}

\title[Operator Daugavet property in semifinite noncommutative $L_1$-spaces]
{The Operator Daugavet Property in Semifinite Noncommutative $L_1$-Spaces}

\author[1]{\fnm{Junxiang} \sur{ Qi}}\email{y25060009@stu.aqnu.edu.cn}

\author*[1]{\fnm{Qi} \sur{Liu}}\email{liuq67@aqnu.edu.cn}
\equalcont{These authors contributed equally to this work.}

\author[2]{\fnm{Yongjin} \sur{Li}}\email{stslyj@mail.sysu.edu.cn}
\equalcont{These authors contributed equally to this work.}

\affil*[1]{\orgdiv{School of Mathematics and Statistics},
	\orgname{Anqing Normal University},
	\orgaddress{
		\city{Anqing},
		\postcode{246133},
		\state{Anhui},
		\country{P. R. China}
}}

\affil[2]{
	\orgdiv{School of Mathematics},
	\orgname{Sun Yat-sen University},
	\orgaddress{
		\city{Guangzhou},
		\postcode{510275},
		\state{Guangdong},
		\country{P. R. China}
}}

\abstract{Let $\mathcal M$ be a diffuse semifinite von Neumann algebra endowed with a faithful normal semifinite trace $\tau$. We prove that, for every nonzero Banach space $Y$, the projective tensor product $L_1(\mathcal M,\tau)\widehat{\otimes}_{\pi}Y$ has the operator Daugavet property. Moreover, the witnessing operators may always be chosen contractive. This extends the operator Daugavet phenomenon from atomless vector-valued $L_1$-spaces to the semifinite noncommutative setting and yields further Daugavet-type consequences for projective symmetric tensor products, all without approximation assumptions.}



\maketitle

\section{Introduction}

The Daugavet equation
\[
\|I+T\|=1+\|T\|
\]
originates from Daugavet's classical work on compact operators on
$C[0,1]$ \cite{Daugavet1963}. A Banach space $X$ is said to have the
Daugavet property if this identity holds for every rank-one operator
$T\colon X\to X$. Classical examples include atomless $L_1$-spaces
and spaces $C(K)$ with no isolated points; see
\cite{KadetsShvidkoySirotkinWerner}. The Daugavet property admits
several geometric characterizations in terms of slices of the unit ball
and is closely related to the geometry of projective tensor products.

The stability of the Daugavet property under projective tensor products
has been studied in a variety of settings. In particular, López-Pérez
and Rueda Zoca \cite{LopezPerezRuedaZoca} obtained tensor-product
results for $L$-embedded Banach spaces under additional density and
metric approximation assumptions. Related developments for
$L$-embedded spaces further exploit bidual geometry and local
reflexivity; see \cite{BecerraMohamed}. These results suggest that
projective tensor products provide a natural framework in which to
study stronger forms of the Daugavet phenomenon. More broadly, structural properties of projective tensor products,
including their local $\ell_p$-structure, type and cotype, and
superreflexivity, have also been extensively studied; see
\cite{BuDiestel,RuedaZoca2025,BuDowling2026,Ryan}.

A parallel theory arises in the noncommutative setting. Let $M$ be a
von Neumann algebra. Its predual $M_*$ is the noncommutative analogue
of an $L_1$-space, and in the semifinite case, when $M$ is equipped
with a faithful normal semifinite trace $\tau$, it is represented by
$L_1(M,\tau)$. Oikhberg \cite{Oikhberg} proved that the predual of a
von Neumann algebra has the Daugavet property precisely when the
algebra is diffuse; see also \cite{BecerraMartin}. Thus diffuseness
plays in the noncommutative setting the role of atomlessness in the
classical $L_1$ theory.

A stronger notion is the \emph{operator Daugavet property} (ODP),
introduced in \cite{RuedaTradaceteVillanueva}. Besides the usual
Daugavet geometry, the ODP requires operators which almost fix a
prescribed finite set while sending a suitable point of a slice to an
arbitrary target. Recently, Huang, Nessipbayev, Sukochev and Xu
\cite{HuangNessipbayevSukochevXu} proved that the predual of every
atomless von Neumann algebra has the operator Daugavet property.

It is therefore natural to ask whether this stronger property is
preserved by projective tensor products in the noncommutative
$L_1$ setting. More precisely, if $M$ is a diffuse semifinite von
Neumann algebra with a faithful normal semifinite trace $\tau$ and
$Y$ is a nonzero Banach space, does
\(
L_1(M,\tau)\widehat{\otimes}_{\pi}Y
\)
have the operator Daugavet property? The purpose of the present paper
is to answer this question affirmatively. Unlike existing approaches
to the ordinary Daugavet property, which typically involve
$L$-embeddedness, bidual arguments, or approximation hypotheses, our
proof relies directly on the projection and corner structure of the
underlying von Neumann algebra.

The present results both extend and strengthen several earlier contributions to the Daugavet theory.  First, Theorem~1.1 extends the operator Daugavet phenomenon from atomless vector-valued $L_1$-spaces, considered in \cite{RuedaTradaceteVillanueva}, to the semifinite noncommutative setting by proving that
\(
L_1(M,\tau)\widehat{\otimes}_{\pi}Y
\)
has the operator Daugavet property for every nonzero Banach space $Y$ whenever $M$ is diffuse.  Within the class of semifinite noncommutative $L_1$-spaces, this also strengthens earlier projective tensor-product results for the ordinary Daugavet property obtained in the $L$-embedded framework \cite{LopezPerezRuedaZoca,BecerraMohamed}, since no density or metric approximation assumption is required and the stronger operator Daugavet property is obtained, with contractive witnessing operators.  Moreover, Corollary~\ref{cor:characterization} refines Oikhberg's characterization of diffuse von Neumann algebras via the Daugavet property \cite{Oikhberg} to an operator-Daugavet and tensor-product characterization in the semifinite case.  Finally, in the commutative setting, the symmetric tensor consequences of Corollary~\ref{cor:symmetric} recover and strengthen the corresponding Daugavet conclusions for vector-valued $L_1$-spaces in \cite{MartinRueda}.

The proof uses a direct noncommutative localization argument based on finite spectral truncation, diffuse equal-trace decompositions, and rectangular corner estimates.  These tools provide simultaneous slice localization, finite-set stability, and exact interpolation, while avoiding the bidual, local reflexivity, and approximation techniques used in earlier tensor-product approaches \cite{LopezPerezRuedaZoca,BecerraMohamed}.  The method also differs from the Kadec--Pe{\l}czy\'nski and submajorization techniques employed in symmetric operator spaces \cite{HuangNessipbayevSukochevXu}.

Beyond its intrinsic geometric interest, the Daugavet property has proved useful in the study of weakly compact and narrow operators, numerical index and lushness, tensor products, polynomial mappings, and the structure of operator algebras; see, for example, \cite{KadetsShvidkoySirotkinWerner,BecerraRodriguez,RuedaTradaceteVillanueva,MartinRueda,BecerraMartin,DantasMartinPerreau}.

Our main result is the following.

\begin{theorem}\label{thm:main}
Let $\mathcal M$ be a diffuse semifinite von Neumann algebra with a faithful normal semifinite trace $\tau$, and let $Y$ be a nonzero Banach space. Then
\[
L_1(\mathcal M,\tau)\widehat\otimes_\pi Y
\]
has the operator Daugavet property. Moreover, the witnessing operators in the definition of the ODP may be chosen with norm at most one.
\end{theorem}

The point is not merely the Daugavet property of the tensor product. That conclusion can also be approached through representability and centralizer methods \cite{BecerraRodriguez}. The stronger ODP assertion gives a uniform interpolation mechanism on finite sets and does not use the metric approximation property. This is relevant in view of recent tensor results for $L$-embedded spaces, where approximation assumptions enter through finite-rank norming and local reflexivity \cite{BecerraMohamed}.

The mechanism behind Theorem~\ref{thm:main} is elementary but genuinely noncommutative. If $e,p\in\mathcal M$ are projections, then
\[
\mathcal B_{e,p}(a)=eap,\qquad \mathcal C_{e,p}(a)=(1-e)a(1-p)
\]
satisfy
\begin{equation}\label{eq:corner-intro}
\|\mathcal B_{e,p}(a)\|_1+\|\mathcal C_{e,p}(a)\|_1\le \|a\|_1.
\end{equation}
Indeed, their sum is the average of $a$ and a two-sided unitary translate of $a$, while the two corners have orthogonal left and right supports. A finite trace corner of a slice element is then partitioned into many equal-trace pieces. Every piece is small on the prescribed finite set, and at least one normalized piece remains in the slice. Replacing that small corner by an arbitrary target produces the desired contractive operator.

All spaces are complex. The closed unit ball and the unit sphere of a Banach space $X$ are denoted by $B_X$ and $S_X$. A slice of $B_X$ is
\[
S(B_X,x^*,\alpha)=\{x\in B_X:\operatorname{Re}x^*(x)>1-\alpha\},\qquad x^*\in S_{X^*},\quad \alpha>0.
\]

\section{Preliminaries}

\begin{definition}[{\cite{RuedaTradaceteVillanueva}}]
A Banach space $X$ has the \emph{operator Daugavet property} if, for every $x_1,\ldots,x_n\in S_X$, every slice $S$ of $B_X$, and every $\varepsilon>0$, there is $x\in S$ such that for each $x'\in B_X$ one can find $T\in\mathcal L(X)$ satisfying
\[
\|T\|\le 1+\varepsilon,\qquad T(x)=x',\qquad \|T(x_j)-x_j\|<\varepsilon\quad (1\le j\le n).
\]
\end{definition}

The ODP implies the weak operator Daugavet property and the Daugavet property. We shall prove the stronger contractive form of Definition~2.1.

For Banach spaces $X$ and $Y$, let $X\otimes Y$ denote the algebraic
tensor product of $X$ and $Y$. The projective tensor product
$X\widehat{\otimes}_{\pi}Y$ is defined as the completion of $X\otimes Y$
with respect to the projective tensor norm
\[
\|z\|_{\pi}
=
\inf\left\{
\sum_{i=1}^{m}\|x_i\|\,\|y_i\|:
z=\sum_{i=1}^{m}x_i\otimes y_i
\right\},
\qquad z\in X\otimes Y.
\]
Let $\mathcal L(X,Y)$ denote the Banach space of all bounded linear
operators from $X$ to $Y$. It is standard that
\[
\bigl(X\widehat{\otimes}_{\pi}Y\bigr)^*
\cong
\mathcal L(X,Y^*)
\]
isometrically; see, for example, \cite{Ryan}. Under this identification,
if $G\in\mathcal L(X,Y^*)$, then
\[
G(x\otimes y)=G(x)(y),
\qquad x\in X,\ y\in Y.
\]
We shall also use the standard identity
\[
B_{X\widehat{\otimes}_{\pi}Y}
=
\overline{\operatorname{conv}}(B_X\otimes B_Y).
\]

Let $\mathcal M$ be semifinite with faithful normal semifinite trace $\tau$. We use the standard realization of $\mathcal M_*$ as $L_1(\mathcal M,\tau)$, with duality
\[
\langle a,b\rangle=\tau(ab),\qquad a\in L_1(\mathcal M,\tau),\ b\in\mathcal M.
\]
Every $a\in L_1(\mathcal M,\tau)$ has a polar decomposition $a=v|a|$ with $v\in\mathcal M$ a partial isometry. Left and right multiplication by contractions of $\mathcal M$ act contractively on $L_1(\mathcal M,\tau)$.

Recall that $\mathcal M$ is \emph{diffuse} if it has no nonzero minimal projection. Every nonzero corner of a diffuse von Neumann algebra is diffuse.

\begin{lemma}\label{lem:finite-spectral}
Let $a\in L_1(\mathcal M,\tau)$ and $\eta>0$. There is a finite-trace spectral projection $q$ of $|a|$ such that
\[
\|a-aq\|_1<\eta.
\]
If $a\ne0$, $q$ may be chosen with $aq\ne0$.
\end{lemma}

\begin{proof}
For $m\ge1$, put $q_m=\mathbf 1_{[1/m,m]}(|a|)$. Since
\[
\frac1m\tau(q_m)\le \tau(|a|q_m)\le \|a\|_1,
\]
we have $\tau(q_m)<\infty$. Moreover, $|a|q_m\to |a|$ in $L_1$ by monotone convergence applied separately at zero and at infinity. Hence $aq_m\to a$ in $L_1$.
\end{proof}

We need a standard diffuse refinement of the abelian algebra generated by a positive operator.

\begin{lemma}\label{lem:diffuse-commuting}
Let $q\in\mathcal M$ be a nonzero finite-trace projection and let $d\in L_1(q\mathcal Mq,\tau)_+$ be supported by $q$. If $\mathcal M$ is diffuse, then there exists a diffuse abelian von Neumann subalgebra $A\subset q\mathcal Mq$ such that every spectral projection of $d$ belongs to $A$. Consequently, for each $N\in\mathbb N$ there are mutually orthogonal projections $p_1,\ldots,p_N\in A$ with
\[
q=p_1+\cdots+p_N,\qquad \tau(p_j)=\frac{\tau(q)}{N}.
\]
\end{lemma}

\begin{proof}
The abelian algebra generated by the spectral projections of $d$ can have atoms only on spectral subspaces on which $d$ is scalar. Each such nonzero spectral corner is diffuse, so inside it one may choose a diffuse abelian von Neumann algebra. Adjoining these algebras to the continuous spectral part gives the required $A$. The restriction of $\tau$ to $A$ is a faithful atomless finite measure; the last assertion is the usual equal-measure partition property.
\end{proof}
\Needspace{18\baselineskip}
	\begin{lemma}\label{lem:absolute-continuity}
		For each finite set $F\subset L_1(M,\tau)$ and each $\eta>0$, there
		exists $\delta>0$ such that
		\[
		\tau(p)<\delta
		\quad\Longrightarrow\quad
		\|ap\|_1+\|pa\|_1<\eta
		\]
		for every projection $p\in M$ and every $a\in F$.
	\end{lemma}
	
	\begin{proof}
		It suffices to consider one $a$. Choose $b\in M\cap L_1(M,\tau)$ with
		$\|a-b\|_1<\eta/4$.\footnote{Here we use the standard fact that
			$M\cap L_1(M,\tau)$ is norm dense in $L_1(M,\tau)$ in the semifinite
			setting.}
		By H\"older's inequality,
		\[
		\|bp\|_1\leq \|bp\|_2\tau(p)^{1/2}
		\leq \|b\|_\infty\tau(p),
		\]
		and the same estimate holds for $pb$. Hence
		\[
		\|ap\|_1+\|pa\|_1
		\leq
		2\|a-b\|_1+2\|b\|_\infty\tau(p),
		\]
		which is less than $\eta$ for sufficiently small $\tau(p)$.
		Take the minimum of the resulting constants over $F$.
	\end{proof}

\section{The operator Daugavet property and consequences}

For projections $e,p\in\mathcal M$, define
\[
\mathcal B_{e,p}(a)=eap,\qquad \mathcal C_{e,p}(a)=(1-e)a(1-p)\qquad (a\in L_1(\mathcal M,\tau)).
\]

\begin{lemma}\label{lem:rectangular}
For every $a\in L_1(\mathcal M,\tau)$,
\begin{equation}\label{eq:corner}
\|\mathcal B_{e,p}(a)\|_1+\|\mathcal C_{e,p}(a)\|_1\le \|a\|_1.
\end{equation}
Consequently, for every Banach space $Y$ and every $z\in L_1(\mathcal M,\tau)\widehat\otimes_\pi Y$,
\begin{equation}\label{eq:corner-tensor}
\|(\mathcal B_{e,p}\otimes I_Y)z\|_\pi+\|(\mathcal C_{e,p}\otimes I_Y)z\|_\pi\le \|z\|_\pi.
\end{equation}
\end{lemma}

\begin{proof}
Put $u=2e-1$ and $w=2p-1$. These are self-adjoint unitaries and
\[
\mathcal B_{e,p}(a)+\mathcal C_{e,p}(a)=\frac12(a+uaw).
\]
Thus the sum of the two corners has norm at most $\|a\|_1$. Moreover,
\[
\mathcal B_{e,p}(a)^*\mathcal C_{e,p}(a)=0=\mathcal C_{e,p}(a)^*\mathcal B_{e,p}(a),
\]
and the positive operators $|\mathcal B_{e,p}(a)|$ and $|\mathcal C_{e,p}(a)|$ have orthogonal supports, respectively dominated by $p$ and $1-p$. Hence
\[
|\mathcal B_{e,p}(a)+\mathcal C_{e,p}(a)|=|\mathcal B_{e,p}(a)|+|\mathcal C_{e,p}(a)|,
\]
which proves \eqref{eq:corner}.

For $z=\sum_{j=1}^m a_j\otimes y_j$, the projective norm gives
\begin{align*}
\|(\mathcal B_{e,p}\otimes I_Y)z\|_\pi+\|(\mathcal C_{e,p}\otimes I_Y)z\|_\pi
&\le \sum_{j=1}^m\bigl(\|\mathcal B_{e,p}(a_j)\|_1+\|\mathcal C_{e,p}(a_j)\|_1\bigr)\|y_j\|\\
&\le \sum_{j=1}^m\|a_j\|_1\|y_j\|.
\end{align*}
Taking the infimum over all representations and then passing to the completion proves \eqref{eq:corner-tensor}.
\end{proof}

The following is the slice-localization lemma.

\begin{lemma}\label{lem:small-slice}
Let $G\in S_{\mathcal L(L_1(\mathcal M,\tau),Y^*)}$, let $\alpha>0$, let $F\subset L_1(\mathcal M,\tau)$ be finite, and let $\eta>0$. There exist $k\in S_{L_1(\mathcal M,\tau)}$, $y_0\in S_Y$, a partial isometry $v\in\mathcal M$, and projections $p,e\in\mathcal M$ such that
\begin{enumerate}
\item[(i)] $\operatorname{Re}G(k,y_0)>1-\alpha$;
\item[(ii)] $k=ekp$ and $e=vpv^*$;
\item[(iii)] $\tau(p)=\tau(e)<\infty$ and
$
\|ap\|_1+\|ea\|_1<\eta\qquad (a\in F);
$
\item[(iv)] $\tau(kv^*)=1$.
\end{enumerate}
\end{lemma}

\begin{proof}
Choose $a\in S_{L_1(\mathcal M,\tau)}$ and $y_0\in S_Y$ so that
\[
\operatorname{Re}G(a,y_0)>1-\alpha/4.
\]
Write $a=v|a|$. By Lemma~\ref{lem:finite-spectral}, after replacing $a$ by a normalized finite spectral truncation and decreasing the error if necessary, we may assume that $q=\operatorname{supp}|a|$ satisfies $0<\tau(q)<\infty$ and
\begin{equation}\label{eq:slice-half}
\operatorname{Re}G(a,y_0)>1-\alpha/2.
\end{equation}
Here $v^*v=q$.

Apply Lemma~\ref{lem:diffuse-commuting} to $|a|$. For an integer $N$ to be chosen, take an equal-trace partition $q=\sum_{j=1}^N p_j$ in the resulting diffuse abelian algebra. Put
\[
\mu_j=\|ap_j\|_1=\tau(|a|p_j),\qquad \nu_j=G(ap_j,y_0).
\]
We have $\sum_j\mu_j=1$, $\sum_j\operatorname{Re}\nu_j>1-\alpha/2$, and $|\nu_j|\le\mu_j$. Therefore, for at least one $j$ with $\mu_j>0$,
\[
\frac{\operatorname{Re}\nu_j}{\mu_j}>1-\alpha/2>1-\alpha.
\]
Set $p=p_j$, $e=vpv^*$, and $k=ap/\mu_j$. Since $p$ commutes with $|a|$, the polar decomposition of $ap$ is $vp|a|p$. It follows that $k=ekp$ and
\[
\tau(kv^*)=\frac{\tau(vp|a|v^*)}{\mu_j}=\frac{\tau(p|a|v^*v)}{\mu_j}=1.
\]
Also $\tau(e)=\tau(p)=\tau(q)/N$. By Lemma~\ref{lem:absolute-continuity}, choosing $N$ sufficiently large makes
\[
\|ap\|_1+\|ea\|_1<\eta
\]
for every $a\in F$. The slice estimate follows from the choice of $p$.
\end{proof}

We now prove the main theorem.

\begin{proof}[Proof of Theorem~\ref{thm:main}]
Put $X=L_1(\mathcal M,\tau)$ and $Z=X\widehat\otimes_\pi Y$. Fix
\[
z_1,\ldots,z_n\in S_Z,\qquad \varepsilon>0,\qquad S=S(B_Z,G,\alpha),
\]
where $G\in S_{Z^*}=S_{\mathcal L(X,Y^*)}$ and $\alpha>0$.

Choose algebraic tensors
\[
w_i=\sum_{j=1}^{m_i}a_{ij}\otimes y_{ij}\qquad (1\le i\le n)
\]
such that $\|z_i-w_i\|_\pi<\gamma$, where $\gamma>0$ will be fixed below. Let
\[
F=\{a_{ij}:1\le i\le n,\ 1\le j\le m_i\}.
\]
Apply Lemma~\ref{lem:small-slice}, with a sufficiently small parameter $\eta$, to obtain $k,y_0,v,p,e$. Write
\[
B=\mathcal B_{e,p},\qquad C=\mathcal C_{e,p}.
\]
Then $k\otimes y_0\in S$.

Choose $y_0^*\in S_{Y^*}$ with $y_0^*(y_0)=1$, and define $H\in Z^*$ by
\[
H(a\otimes y)=\tau(av^*)y_0^*(y).
\]
Clearly $\|H\|\le1$. For an arbitrary $z'\in B_Z$, define
\begin{equation}\label{eq:phi}
\Phi_{z'}(z)=(C\otimes I_Y)z+H\bigl((B\otimes I_Y)z\bigr)z',\qquad z\in Z.
\end{equation}
Lemma~\ref{lem:rectangular} gives
\begin{align*}
\|\Phi_{z'}(z)\|_\pi
&\le \|(C\otimes I_Y)z\|_\pi+\|(B\otimes I_Y)z\|_\pi\|z'\|_\pi\\
&\le \|z\|_\pi.
\end{align*}
Thus $\|\Phi_{z'}\|\le1$. Since $Bk=k$, $Ck=0$, and $H(k\otimes y_0)=\tau(kv^*)y_0^*(y_0)=1$, we have
\[
\Phi_{z'}(k\otimes y_0)=z'.
\]

It remains to verify the almost-fixing condition. For $a\in F$,
\[
\|a-Ca\|_1=\|ea+(1-e)ap\|_1\le\|ea\|_1+\|ap\|_1,
\]
\[
\|Ba\|_1\le\min\{\|ea\|_1,\|ap\|_1\}.
\]
Consequently, by choosing $\eta$ small enough in Lemma~\ref{lem:small-slice},
\[
\|(I-C)\otimes I_Y(w_i)\|_\pi+\|B\otimes I_Y(w_i)\|_\pi<\varepsilon/2\qquad (1\le i\le n).
\]
Since $B$ and $C$ are contractions,
\begin{align*}
\|\Phi_{z'}(z_i)-z_i\|_\pi
&\le \|((C-I)\otimes I_Y)z_i\|_\pi+\|(B\otimes I_Y)z_i\|_\pi\\
&\le 3\gamma+\|((C-I)\otimes I_Y)w_i\|_\pi+\|(B\otimes I_Y)w_i\|_\pi.
\end{align*}
Taking $\gamma<\varepsilon/6$ makes the last expression smaller than $\varepsilon$. The point $k\otimes y_0\in S$ is independent of $z'$, while \eqref{eq:phi} works for every $z'\in B_Z$. This is precisely the ODP.
\end{proof}

\begin{remark}
The proof gives more than required: the witnessing point is elementary and every target is reached exactly by a contraction. No finite-rank approximation, local reflexivity, or approximation property of either factor is used.
\end{remark}

Combining Theorem~\ref{thm:main} with Oikhberg's characterization of the Daugavet property of von Neumann preduals gives a clean algebraic characterization.

\begin{corollary}\label{cor:characterization}
Let $\mathcal M$ be a semifinite von Neumann algebra with faithful normal semifinite trace $\tau$. The following assertions are equivalent:
\begin{enumerate}
\item[(i)] $\mathcal M$ is diffuse;
\item[(ii)] $L_1(\mathcal M,\tau)$ has the ODP;
\item[(iii)] for every nonzero Banach space $Y$, the space $L_1(\mathcal M,\tau)\widehat\otimes_\pi Y$ has the ODP.
\end{enumerate}
\end{corollary}

\begin{proof}
Theorem~\ref{thm:main} gives (i)$\Rightarrow$(iii), and (iii)$\Rightarrow$(ii) follows by taking $Y=\mathbb C$. Finally, the ODP implies the Daugavet property, whereas $L_1(\mathcal M,\tau)=\mathcal M_*$ has the Daugavet property only if $\mathcal M$ is diffuse \cite{Oikhberg}. Thus (ii)$\Rightarrow$(i).
\end{proof}

\begin{corollary}\label{cor:dual}
For every nonzero Banach space $Y$, the dual space
\[
\mathcal L(L_1(\mathcal M,\tau),Y^*)
\]
has a weak-star octahedrality consequence dual to the Daugavet property of $L_1(\mathcal M,\tau)\widehat\otimes_\pi Y$.
\end{corollary}

The ODP implies the WODP. Recent results show that the WODP implies its polynomial version and transfers to projective symmetric tensor powers \cite{DantasMartinPerreau}. Thus we obtain:

\begin{corollary}\label{cor:symmetric}
Under the assumptions of Theorem~\ref{thm:main}, let $Y\ne\{0\}$ and $N\in\mathbb N$. Then the $N$-fold projective symmetric tensor product
\[
\widehat\otimes_{\pi,s,N}\bigl(L_1(\mathcal M,\tau)\widehat\otimes_\pi Y\bigr)
\]
has the weak operator Daugavet property and, in particular, the Daugavet property.
\end{corollary}

\bmhead{Acknowledgements}

Thanks to all the members of the Functional Analysis Research Team at the School of Mathematics and Statistics, Anqing Normal University, for their valuable discussions and corrections regarding the
challenges and errors encountered in this article. 

\section*{Declarations}

\begin{itemize}
	\item	Funding: This work received no funding support.
	\item Conflict of interest: The author declares no conflict of
	interest.
	\item Data availability: This is a purely theoretical study without any experimental. All conclusions are obtained via mathematical analysis and proof.
	
	\item Author contributions:
	Junxiang Qi, Qi Liu, and Yongjin Li contributed equally to the conceptualization, theoretical development, proofs, manuscript preparation, and revision of this work. All authors discussed the results, reviewed the manuscript, and approved the final version.
	
\end{itemize}

\backmatter

%
% ...
%
% ...

\end{document}